\documentclass[11pt]{article}
\usepackage{amsmath,amssymb}

\newtheorem{propo}{{\bf Proposition}}[section]
\newtheorem{coro}[propo]{{\bf Corollary}}
\newtheorem{lemma}[propo]{{\bf Lemma}} 
\newtheorem{theor}[propo]{{\bf Theorem}} 
\newtheorem{ex}{{\sc Example}}[section]

\newenvironment{proof}{{\bf Proof.}}{$\Box$}

\def\N{{\mathbb N}}

\begin{document}

\vspace*{1.0in}

\begin{center} ON THE LENGTHS OF CERTAIN CHAINS OF SUBALGEBRAS IN LIE ALGEBRAS  
\end{center}
\bigskip

\begin{center} DAVID A. TOWERS 
\end{center}
\bigskip

\begin{center} Department of Mathematics and Statistics

Lancaster University

Lancaster LA1 4YF

England

d.towers@lancaster.ac.uk 
\end{center}
\bigskip

\begin{abstract} In this paper we study the lengths of certain chains of subalgebras of a Lie algebra $L$: namely, a chief series, a maximal chain of minimal length, a chain of maximal length in which each subalgebra is modular in $L$, and a chain of maximal length in which each subalgebra is a quasi-ideal of $L$. In particular we show that, over a field $F$ of characteristic zero, a Lie algebra $L$ with radical $R$ has a maximal chain of subalgebras and a chain of subalgebras all of which are modular in $L$ of the same length if and only if $L = R$, or$\sqrt{F} \not \subseteq F$ and $L/R$ is a direct sum of isomorphic three-dimensional simple Lie algebras.  
\end{abstract}

\section{Introduction}
It has been shown by a number of authors that lattice-theoretic information about the subalgebra lattice of a Lie algebra can be used to infer information about the structure of the algebra itself. These studies resemble similar ones in the theory of groups, but there are interesting and striking differences. The current paper is inspired by the lattice-theoretic characterisations of finite solvable groups given by Shareshian and Woodroofe in \cite{sw}, but again most of the results obtained, and the methods used, are different.

Throughout $L$ will denote a finite-dimensional Lie algebra over a field $F$. We write $minmax(L)$ for the minimal length of a maximal chain of subalgebras of $L$ and $\ell(L)$ for the length of a chief series for $L$. In section two we consider the relationship between $\ell(L)$ and $\ell(M)$ when $M$ is a solvable maximal subalgebra of $L$. The main result is that if $L$ is nonsolvable and $F$ has characteristic zero, or else is algebraically closed of characteristic greater than 5, then $\ell(M) \geq \ell(L)$. In the case where $F$ has characteristic zero the nonsolvable algebras $L$ for which $\ell(M) = \ell(L)$ are characterised; in particular, if $F$ is also algebraically closed there are no such algebras.  
\par

The purpose of section three is to investigate the relationship between $\ell(L)$ and $minmax(L)$. We show first that if $L$ is solvable then there is a maximal chain of subalgebras which has the same length as a chief series for $L$ and such that the dimensions of the factor spaces in the two chains are in a one-one correspondence. Moreover, if $L$ is solvable then $minmax(L) = \ell(L)$. We then consider whether, as for groups, the converse is true. This is shown to be the case if $F$ has characteristic zero or is algebraically closed of characteristic greater than five. However, the situation for Lie algebras differs from that for groups in a number of repects: for example, there are Lie algebras $L$ with $minmax(L) = \ell(L) + 1$, whereas there is no group with this property. For all algebras $L$ we have $minmax(L) \geq \ell(L)$. We show that if $L$ is nonsolvable and $F$ has characteristic zero, or else is algebraically closed of characteristic greater than 5, then $minmax(L) \geq \ell(L) + 1$. In the case where $F$ has characteristic zero the algebras $L$ for which $minmax(L) = \ell(L) + 1$ are characterised; in particular, if $F$ is also algebraically closed, again there are no such algebras.   
\par

A subalgebra $U$ of $L$ is called {\em modular} in $L$ if it is a modular element in the lattice of subalgebras of $L$; that is, if
$$ \langle U,B \rangle \cap C = \langle B, U \cap C \rangle  \hspace{.3in} \hbox{for all subalgebras}\hspace{.1in} B \subseteq C,
$$
and
$$ \langle U,B \rangle \cap C =  \langle B \cap C,U \rangle  \hspace{.3in} \hbox{for all subalgebras}\hspace{.1in} U \subseteq C,
$$
(where, $ \langle U, B \rangle$ denotes the subalgebra of $L$ generated by $U$ and $B$). A subalgebra $Q$ of $L$ is called a {\em quasi-ideal} of $L$ if $[Q,V] \subseteq Q + V$ for every subspace $V$ of $L$. We write $mod\ell(L)$ for the maximal length of a chain of subalgebras each of which is modular in $L$, and $qi\ell(L)$ for the maximal length of a chain of subalgebras each of which is a quasi-ideal of $L$. 
\par

In section three we consider the relationship between these two measures and $\ell(L)$. It is shown that, over any field, $qi\ell(L) = \ell(L), \ell(L) + 1$ or $\ell(L) + 2$, and that if the field has characteristic zero, or if $L$ is restricted and $F$ is algebraically closed of characteristic $p > 0$, then $mod\ell(L) = \ell(L)$ or $\ell(L)+1$. Furthermore, the algebras $L$ for which $qi(L) \neq \ell(L)$ or $mod\ell(L) \neq \ell(L)$ are described.  This situation again differs from the situation for groups, where $mod\ell(G) = \ell(G)$ for every finite group $G$. It is shown that, over a field $F$ of characteristic zero, Lie algebras $L$ that are solvable or whose Levi subalgebra is a direct sum of isomorphic three-dimensional simple Lie algebras with a one-dimensional maximal subalgebra, are characterised by the purely lattice-theoretic condition that $mod\ell(L) = minmax(L)$. Finally it is shown that, if $L$ be a restricted Lie algebra over an algebraically closed field $F$ of characteristic $p > 0$, then $mod\ell(L) = \ell(L) + 1$ if and only if $L$ has an ideal $B$ such that $L/B \cong sl_2(F)$ or the Witt algebra $W(1:\underline{1})$.
\par

If $A$ and $B$ are subalgebras of $L$ for which $L = A + B$ and $A \cap B = 0$ we will write $L = A \dot{+} B$; if, furthermore, $A, B$ are ideals of $L$ we write $L = A \oplus B$.  

\section{Chief series and solvable maximal subalgebras}
Let $0 = L_0 < L_1 < \ldots L_n = L$ be a chief series for $L$. Then we put $\ell(L) = n =$ the {\em length} of the chief series. For an ideal $B$ of $L$, let $il_L(B)$ denote the largest number $r$ such that there is a chain $0 = B_0 < \ldots < B_r= B$ of ideals of $L$ of length $r$. Clearly $\ell(L) = \ell(L/B) + il_L(B)$; in fact, if $C$ is an ideal of $L$ with $C \subseteq B$ we have $i_L(B) = i_{L/C}(B/C) + i_L(C)$. If $U$ is a subalgebra of $L$ we define the {\em core} (with respect to $L$) of $U$, $U_L$, to be the largest ideal of $L$ contained in $U$.

\begin{lemma}\label{l:elmax} Let $M$ be a maximal subalgebra of $L$ and let $B$ be an ideal of $L$ such that $B/M_L$ is a minimal ideal of $L/M_L$. Then
$$ \ell(M) - \ell(L) = il_M(M_L) - il_L(M_L) + il_{M/M_L}((M \cap B)/M_L) - 1.$$
\end{lemma}
\begin{proof} Clearly $L = M + B$, so $L/B \cong M/(M \cap B)$. Now
$$\begin{array}{lll}
\ell(M) - \ell(L) & = & \ell(M/M_L) + il_M(M_L) - \ell(L/M_L) - il_L(M_L) \\ 
 & = & il_M(M_L) - il_L(M_L) + \ell(M/M_L) - (1 + \ell(L/B)) \\
 & = & il_M(M_L) - il_L(M_L) + [\ell(M/M_L) - \ell(M/(M \cap B))] - 1 \\
 & = & il_M(M_L) - il_L(M_L) + il_{M/M_L}((M \cap B)/M_L) - 1.
\end{array}$$
\end{proof}

\begin{coro}\label{c:elmax} Let $M$ be a maximal subalgebra of $L$ and let $B$ be an ideal of $L$ such that $B/M_L$ is a minimal ideal of $L/M_L$. Then
\begin{itemize}
\item[(i)] $\ell(M) \geq \ell(L) - 1$;
\item[(ii)] if $M \cap B \neq M_L$, then $\ell(M) \geq \ell(L)$; and
\item[(iii)] if $(M \cap B)/M_L$ is neither trivial nor a minimal ideal of $M/M_L$, then $\ell(M) \geq \ell(L) + 1$
\end{itemize}
\end{coro}
\begin{proof} Simply note that $il_L(M_L) \leq il_M(M_L)$.
\end{proof}
\bigskip

The {\em Frattini ideal} of $L$, $\phi(L)$, is the largest ideal of $L$ contained in all maximal subalgebras of $L$. The {\em abelian socle} of $L$, Asoc $L$, is the sum of the minimal abelian ideals of $L$. If $B/C$ is a chief factor of $L$ we define the {\em centraliser} in $L$ of $B/C$ to be $C_L(B/C) = \{x\in L : [x,B] \subseteq C \}$.

\begin{lemma}\label{l:max} Let $L$ be a Lie algebra with nilradical $N$, and let $M$ be a maximal subalgebra of $L$ with $N \not \subseteq M$. Then $il_M(\phi(L)) = il_L(\phi(L))$ and $\ell(M) = \ell(L) - 1$.
\end{lemma}
\begin{proof} We have $N \neq \phi(L)$ since $N \not \subseteq M$. Let $A/\phi(L) \subseteq N/\phi(L)$ be a minimal ideal of $L/\phi(L)$ with $A \not \subseteq M$. Such an $A$ exists by \cite[Theorem 7.4]{frat}. Then $L = A + M$, $A \cap M = \phi(L)$ and $A \subseteq N$.
\par

Let $B/C$ be a chief factor of $L$ with $B \subseteq \phi(L)$. Then $N$ is the intersection of the centralizers of the factors in a chief series for $L$, by \cite[Lemma 4.3]{bg}, so $A \subseteq N \subseteq C_L(B/C)$. It follows that $B/C$ is a chief factor of $M$ and $il_M(\phi(L)) = il_L(\phi(L))$. In view of this, to show that $\ell(M) = \ell(L) - 1$ we can assume that $\phi(L) = 0$. But then $L = A \dot{+} M$ and if $B$ is an ideal of $M$, $A + B$ is an ideal of $L$. The result follows.
\end{proof}

\begin{lemma}\label{l:ell} Let $L$ be a Lie algebra, over a field of characteristic zero, with radical $R$, nilradical $N$, and Levi decomposition $L = R \dot{+} S$. Then 
\[ \ell(L) = il_L(\phi(L)) + il_{L/\phi(L)}(N/\phi(L)) + \dim(R/N) + \ell(S).
\]
\end{lemma}
\begin{proof} It suffices to show that if $\phi(L) = 0$ then 
\[ \ell(L) = il_{L}(N) + \dim(R/N) + \ell(S).
\]
This follows from the fact if $\phi(L) = 0$ then $L = N \dot{+} (S \oplus C)$, where $N =$ Asoc $L$ and $C$ is an abelian subalgebra of $L$, by \cite[Theorems 7.4, 7.5]{frat}.
\end{proof} 
\bigskip

We will call a simple Lie algebra $L$ over a field $F$ of characteristic different from two {\em special} if it  has a one-dimensional maximal subalgebra. It is well-known that $L$ is special if and only if it is three-dimensional and $\sqrt{F} \not \subseteq F$.

\begin{lemma}\label{l:simple} Let $S = S_1 \oplus \ldots \oplus S_n$ be a semisimple Lie algebra over a field of characteristic zero, where $S_i$ is a simple ideal of $S$ for $1 \leq i \leq n$, and $n>1$. Let $M$ be a maximal subalgebra of $S$ with $S_n \not \subseteq M$. 
\begin{itemize}
\item[(i)] If $S_j \not \subseteq M$ for some $1 \leq j \leq n-1$ then $S_n \cong S_j$ and $\ell(M) = \ell(S) -1$.
\item[(ii)] If $S_j \subseteq M$ for all $1 \leq j \leq n-1$ then $\ell(M) \geq \ell(S)$ with equality if and only if either
\begin{itemize}
\item[(a)] $S_n$ is three-dimensional simple and $\sqrt{F} \not \subseteq F$, or
\item[(b)] $M \cap S_n$ is simple.
\end{itemize}
\end{itemize}
\end{lemma}
\begin{proof} (i) We have $S = M + S_n = M + S_j$ and $M \cap S_n = M \cap S_j = 0$, since each of these is an ideal of $S$. It follows that $M \cong L/S_n \cong L/S_j$, whence $S_n \cong S_j$ and $\ell(M) = n-1 = \ell(S) - 1$.
\medskip

\noindent (ii) In this case $M = S_1 \oplus \ldots \oplus S_{n-1} \oplus M \cap S_n$, whence $\ell(M) \geq n = \ell(S)$. Moreover, $\ell(M) = \ell(S)$ if and only if $\dim M \cap S_n = 1$, in which case $S_n$ is special (case (a)) or $M \cap S_n$ is simple (case (b)).
\end{proof}

\begin{propo}\label{p:nonsolv} Let $L$ be a non-solvable Lie algebra over field $F$ and let $M$ be a solvable maximal subalgebra of $L$.
\begin{itemize}
\item[(i)] If $F$ has characteristic zero, or else is algebraically closed of characteristic greater than 5,  then $\ell(M) \geq \ell(L)$.
\item[(ii)] If $F$ has characteristic zero then $\ell(M) = \ell(L)$ if only if $\sqrt{F} \not \subseteq F$, $L/R \cong S$ is three-dimensional simple, where $R$ is the radical of $L$, and each chief factor $B/C$ with $\phi(L) \subseteq C \subset B \subseteq N$ is an irreducible $M \cap S$-module.
\end{itemize}
\end{propo}
\begin{proof} (i) Let $L$ be a minimal counter-example. If $M_L \neq 0$ then $L/M_L$ is non-solvable and $\ell(M/M_L) \geq \ell(L/M_L)$ by the minimality of $L$. But now $\ell(M) - il_M(M_L) \geq \ell(L) - il_L(M_L)$, whence $\ell(M) \geq  \ell(L) + il_M(M_L) - il_L(M_L) \geq \ell(L)$, a contradiction. Thus $M_L = 0$.
\par

Let $B$ be a minimal ideal of $L$. Then $B \not \subseteq M$, so $L = M + B$. We must have that $M \cap B$ is trivial or is a minimal ideal of $M$, since otherwise Corollary \ref{c:elmax} (iii) is contradicted. Suppose first that it is a minimal ideal of $M$. Then
$\ell(M) = \ell(M/(M \cap B)) + 1 = \ell(L/B) + 1 = \ell(L)$, a contradiction. We therefore have that $M \cap B = 0$. But then $M$ is a c-ideal of $L$, and so $L$ is solvable, by \cite[Theorems 3.2 and 3.3]{cid}, a contradiction.
\medskip

\noindent (ii) We show first that if $L/R$ does not have a one-dimensional maximal subalgebra then $\ell(M) > \ell(L)$. Let $L$ be a minimal counter-example. Suppose $M_L \neq 0$. Clearly $M_L = R$ so $L/M_L$ is non-solvable and does not have a one-dimensional maximal subalgebra. We thus have $\ell(M/M_L) > \ell(L/M_L)$ by the minimality of $L$. But then, as above, $\ell(M) > \ell(L)$, a contradiction, so $M_L = 0$ and $L$ is semisimple. If $B$ is a simple ideal of $L$, then $L = M + B$ and $L/B \cong M/M \cap B$ is solvable, so $L$ is simple. But then $\ell (M) = \ell(L) = 1$ and so $M$ is one-dimensional, a contradiction. 
\par

It follows that, if $\ell(M) = \ell(L)$, then $\sqrt{F} \not \subseteq F$, $L/R$ is three-dimensional simple and $M = R + Fs$ for some $s \in S$. But now $$\ell(M) = il_M(\phi(L)) + il_{M/\phi(L)}(N/\phi(L)) + \dim(R/N) + 1,$$ by Lemma \ref{l:ell}. It follows from Lemma \ref{l:ell} that $\ell(L) = \ell(M)$ if and only if $il_{M/\phi(L)}(N/\phi(L)) = il_{L/\phi(L)}(N/\phi(L))$. This occurs precisely when each chief factor $B/C$ with $\phi(L) \subseteq C \subset B \subseteq N$ is an irreducible $Fs$-module.
\end{proof}
\bigskip

Note that, in particular, the chief factors referred to in part (ii) of the above result must have dimension at most two, as the following example illustrates.

\begin{ex}\label{e:one} Let $L = A \rtimes S$ be the semidirect product of an abelian ideal, $A$, and a three-dimensional non-split simple Lie algebra, $S$, over the real field. Let $S$ act irreducibly on $A$ and let $\dim A \geq 3$. Then $\ell(L) = 2$. However, for any $s \in S$, $A$ has an ad\,$s$-invariant subspace of dimension at most two (see, for example, \cite[page23]{gein}). It follows that the maximal subalgebra $M = A \rtimes Fs$ has $\ell(M) \geq 3$.  
\end{ex}

\begin{coro}\label{c:nonsolv} Let $L$ be a nonsolvable Lie algebra over a field $F$ which is algebraically closed of characteristic zero, and let $M$ be a maximal solvable subalgebra of $L$. Then $\ell(M) \geq \ell(L) + 1$.
\end {coro}

\section{Chains of subalgebras of maximal length} 
The following two results are analogues of \cite[Theorems 1,2]{koh}.

\begin{theor}\label{t:exists} Let $L$ be a solvable Lie algebra with chief series $L = L_r > L_{r-1} > \ldots > L_0 = 0$. Then there is a maximal chain of subalgebras $L = M_r > M_{r-1} > \ldots > M_0 = 0$ of $L$ and a permutation $\pi \in S_r$ such that dim\,$(M_i/M_{i-1}) =$ dim\,$(L_{\pi(i)}/L_{\pi(i)-1})$ for $1 \leq i \leq r$.
\end{theor}
\begin{proof} We use induction on dim\,$L$. Let $A/\phi(L)$ be a minimal ideal of $L/\phi(L)$. Then there is a maximal subalgebra $M$ of $L$ with $L = A + M$ and $A \cap M = \phi(L)$. Let $B/C$ be a chief factor of $L$ with $B \subseteq \phi(L)$. Then, as in Lemma \ref{l:max}, $B/C$ is a chief factor of $M$. Form a chief series $L = L_r > \ldots > L_k = A > L_{k-1} = \phi(L) > \dots L_0 = 0$ for $L$. Since $M/\phi(L) \cong L/A$ and the chief factors of $L$ inside $\phi(L)$ are chief factors of $M$, the chief factors of $M$ are isomorphic to $L_j/L_{j-1}$ for $j \in \{0, \ldots, k-1\} \cup \{k+1, \ldots, r\}$. It follows by induction that $M$ has a maximal chain of subalgebras such that there is a bijection from the set of codimensions of successive elements in this chain to the set of dimensions of these chief factors. But dim\,$L/M =$ dim\,$A/\phi(L)$, which completes the proof. 
\end{proof}

\begin{theor}\label{t:maxmin} Let $L$ be a solvable Lie algebra with chief series $L = L_s > L_{s-1} > \ldots L_0 = 0$, and let $L = M_r > M_{r-1} > \ldots > M_0 = 0$ be a maximal chain of subalgebras of $L$ of minimum length. Then $r = s$ and there is a permutation $\pi \in S_r$ such that dim\,$(M_i/M_{i-1}) =$ dim\,$(L_{\pi(i)}/L_{\pi(i)-1})$ for $1 \leq i \leq r$.
\end{theor}
\begin{proof} We have $\ell(L) = s \geq r$, by Theorem \ref{t:exists}. We use induction on dim\,$L$. Then there is a bijection from the set of codimensions of successive elements in the chain $M_{r-1} > \ldots > M_0 = 0$ to the set of dimensions of the chief factors of $M_{r-1}$. Hence $\ell(M_{r-1}) = r-1$. Let $B$ be the core of $M_{r-1}$, and $A/B$ be a chief factor of $L$. Then $L = M_{r-1} + A$ and $B = M_{r-1} \cap A$.
\par

Form a chief series $L = L_s > \ldots > L_k = A > L_{k-1} = B > \dots L_0 = 0$ for $L$. Then $L/A \cong M_{r-1}/B$, so the chief factors of $M_{r-1}$ containing $B$ are isomorphic to $L_j/L_{j-1}$ for $k+1 \leq j \leq s$. Also any chief factor of $L$ inside $B$ is also a chief factor of $M_{r-1}$, since otherwise $\ell(M_{r-1}) \geq \ell(L) \geq r$, a contradiction. Finally dim\,$L/M_{r-1} =$ dim\,$A/B$, which completes the proof.
\end{proof}
\bigskip

Put $minmax(L) =$ the minimum length of a maximal chain of subalgebras of $L$. 

\begin{lemma}\label{l:minmaxell} For every Lie algebra $L$ we have $minmax(L) \geq \ell(L)$.
\end{lemma}
\begin{proof} Let $L$ be a minimal counter-example, and let $$0 = M_0 < \ldots < M_r = L$$ be a maximal chain of subalgebras of $L$. Clearly $r \geq 1$. Then
$$\begin{array}{llllll} 
r & \geq & 1 + minmax(M_{r-1}) & \geq & 1 + \ell(M_{r-1}) & \hbox{ by the minimality of } L\\
 & & & \geq & \ell(L) & \hbox{ by Corollary \ref{c:elmax} (i)}.
 \end{array}$$
\end{proof}

\begin{propo}\label{p:char0} Let $L$ be a Lie algebra, over a field of characteristic zero, with radical $R$ and Levi decomposition $L = R \dot{+} S$. Then $minmax(L) \leq il_L(R) + minmax(S)$.
\end{propo}
\begin{proof} We use induction on $\dim R$. The result is clear if $R = 0$, so suppose that $R \neq 0$. Clearly $\phi(L) \neq R$ by Levi's Theorem, so let $A/\phi(L)$ be a minimal ideal of $L/\phi(L)$ with $A \subseteq R$. Then there is a maximal subalgebra $M$ of $L$ such that $L = A + M$, $S \subseteq M$ and $A \cap M = \phi(L)$. Let $R(M)$ be the radical of $M$. Then $A + R(M)$ is a solvable ideal of $L$, and so $R(M) \subseteq R \cap M \subseteq R(M)$. By induction we have that 
\begin{align}
minmax(M) & \leq il_M(R \cap M) + minmax(S) \nonumber \\
   & = il_M(\phi(L)) + il_{M/\phi(L)}((R \cap M)/\phi(L)) + minmax(S) \nonumber \\
   & = il_L(\phi(L)) + il_{L/\phi(L)}(R) - 1 + minmax(S) \nonumber \\
   & = il_L(R) - 1 + minmax(S), \nonumber
\end{align}
using Lemma \ref{l:max}. But now $$minmax(L) \leq minmax(M) + 1 \leq il_L(R) + minmax(S).$$
\end{proof}

\begin{theor}\label{t:nonsolv} Let $L$ be a nonsolvable Lie algebra over a field $F$ which is either of characteristic zero or else is algebraically closed of characteristic greater than 5. Then $minmax(L) \geq \ell(L) + 1$.
\end {theor}
\begin{proof} Let $L$ be a minimal counter-example, and let $0 = M_0 < \ldots < M_r = L$ be a maximal chain of subalgebras of $L$ of minimal length $r$. Clearly $r \geq 1$. If $M_{r-1}$ is not solvable we get
$$\begin{array}{llllll} 
r - 1 & \geq & minmax(M_{r-1}) & \geq & 1 + \ell(M_{r-1}) & \hbox{ by the minimality of } L\\
 & & & \geq & \ell(L) & \hbox{ by Corollary \ref{c:elmax} (i)},
 \end{array}$$
a contradiction.
\par

If $M_{r-1}$ is solvable then
$$\begin{array}{llll} 
r - 1 & \geq & minmax(M_{r-1}) = \ell(M_{r-1}) & \hbox{ by Theorem \ref{t:maxmin} }\\
  & \geq & \ell(L) & \hbox{ by Proposition \ref{p:nonsolv} },
 \end{array}$$
a contradiction.
\end{proof} 
\bigskip

Putting together Theorems \ref{t:maxmin} and \ref{t:nonsolv} we have the following corollary.

\begin{coro}\label{c:solvmaxmin} Let $L$ be a Lie algebra over a field over field $F$ which is either of characteristic zero or else is algebraically closed of characteristic greater than 5. Then $minmax(L) = \ell(L)$ if and only if $L$ is solvable.
\end{coro}

\begin{coro}\label{c:nonsolvb} Let $L$ be a nonsolvable Lie algebra over a field $F$ which is algebraically closed of characteristic zero. Then $minmax(L) \geq \ell(L) + 2$.
\end {coro}
\begin{proof} We proceed as in the proof of Theorem \ref{t:nonsolv}. Let $L$ be a minimal counter-example, and let $0 = M_0 < \ldots < M_r = L$ be a maximal chain of subalgebras of $L$ of minimal length $r$. Clearly $r \geq 1$. If $M_{r-1}$ is not solvable we get
$$\begin{array}{llllll} 
r - 1 & \geq & minmax(M_{r-1}) & \geq & 2 + \ell(M_{r-1}) & \hbox{ by the minimality of } L\\
 & & & \geq & 1 + \ell(L) & \hbox{ by Corollary \ref{c:elmax} (i)},
 \end{array}$$
a contradiction.
\par

If $M_{r-1}$ is solvable then
$$\begin{array}{llll} 
r - 1 & \geq & minmax(M_{r-1}) = \ell(M_{r-1}) & \hbox{ by Theorem \ref{t:maxmin} }\\
  & \geq & \ell(L) + 1 & \hbox{ by Corollary \ref{c:nonsolv} },
 \end{array}$$
a contradiction.
\end{proof} 
\bigskip

The above result mirrors more closely what happens in group theory. It is noted in \cite{sw} that there is no group $G$ with $minmax(G) = \ell(G) + 1$. However, if $L$ is three-dimensional non-split simple, then $\ell(L) = 1$ and $minmax(L) = 2$. In fact, when $F$ has characteristic zero, we can classify the algebras $L$ for which $minmax(L) = \ell(L) + 1$. 

\begin{propo}\label{p:split} Let $L$ be a Lie algebra over a field $F$ of characteristic zero with radical $R$. Then $minmax(L) = \ell(L) + 1$ if and only if $\sqrt{F} \not \subseteq F$ and $L/R$ is a direct sum of isomorphic three-dimensional simple Lie algebras.
\end{propo}
\begin{proof} We establish the `only if' first. By Theorem \ref{t:nonsolv} and Corollary \ref{c:solvmaxmin} it suffices to show that if $L/R$ is not a direct sum of isomorphic three-dimensional special simple Lie algebras then $minmax(L) \geq \ell(L) + 2$. Let $L$ be a minimal counter-example, and let
\begin{align}
0 = M_0 < M_1 < \ldots < M_n = L \nonumber
\end{align}
be a maximal chain of minimal length. If $R \not \subseteq M_{n-1}$ then $L = R + M_{n-1}$ and $L/R \cong M_{n-1}/(R \cap M_{n-1})$, so $M_{n-1}/R(M_{n-1})$ is not a direct sum of isomorphic three-dimensional special simple Lie algebras (where $R(M_{n-1})$ is the radical of $M_{n-1}$). It follows that
\begin{align}
n-1 \geq minmax(M_{n-1}) \geq \ell(M_{n-1}) + 2 \geq \ell(L) + 1,
\end{align}
by the inductive hypothesis and Corollary \ref{c:elmax}, whence $n \geq \ell(L) + 2$, as claimed.
\par

So suppose that $R \subseteq M_{n-1}$. Let $L = R \dot{+} S$ be the Levi decomposition of $L$, let $S = S_1 \oplus \ldots \oplus S_r$ where $S_i$ is a simple ideal of $L$ for $1 \leq i \leq r$, and suppose that $S_r \not \subseteq M_{n-1}$. If $M_{n-1}$ is solvable, then
\begin{align}
n-1 \geq minmax(M_{n-1}) \geq \ell(L) + 1, \nonumber
\end{align}
by Proposition \ref{p:nonsolv}, and the result follows. If $M_{n-1}/R(M_{n-1})$ is non-zero but not a direct sum of isomorphic three-dimensional special simple Lie algebras, then $(1)$ holds again. So suppose that $M_{n-1}/R(M_{n-1})$ is a direct sum of isomorphic three-dimensional special simple Lie algebras. Then $S_j \subseteq M_{n-1}$ for $1 \leq j \leq r$, since otherwise $S$ is a direct sum of isomorphic three-dimensional special simple Lie algebras, by Lemma \ref{l:simple}. We therefore have
$$\begin{array}{llll}
n-1 & \geq minmax(M_{n-1}) \geq \ell(M_{n-1}) + 1 & \hbox{by Theorem \ref{t:nonsolv}} \\ 
   & = il_{M_{n-1}}(R) + \ell(M_{n-1} \cap S) + 1 &  \\
    & \geq il_L(R) + \ell(M_{n-1} \cap S) + 1  &  \\
    & \geq il_L(R) + \ell(S) + 1 & \hbox{by Lemma \ref{l:simple}} \\
    & =  \ell(L) + 1, &
\end{array}$$
and the result follows.
\par

Suppose now that $L/R$ is a direct sum of isomorphic three-dimensional special simple Lie algebras. As a result of Theorem \ref{t:nonsolv} it suffices to show that $minmax(L) \leq \ell(L) + 1$. Moreover, by Proposition \ref{p:char0}, it suffices to show that $minmax(L/R) \leq \ell(L/R) + 1$. Let $L/R \cong S_1 \oplus \ldots \oplus S_k$, where the $S_i$ are isomorphic copies of the same three-dimensional special simple Lie algebra. Define $\Delta_i$ inductively by $\Delta_1 = \{s + \bar{s}: s \in S_{k-1}, \bar{s} = \theta(s),$ where $\theta$ is an isomorphism from $S_{k-1}$ to $S_k \}$ (the diagonal subalgebra of $S_{k-1} \oplus S_k$), $\Delta_i =$ the diagonal subalgebra of $S_{k-i} \oplus \Delta_{i-1}$ for $2 \leq i \leq k-1$. Then 
$$0 < Fs < \Delta_{k-1} < S_1 \oplus \Delta_{k-2} \ldots < S_1 \oplus \ldots \oplus S_{k-2} \oplus \Delta_1 < L$$ is a maximal chain of subalgebras of $L$, so $minmax(L) \leq k+1$. But then $minmax(L) = \ell(L) + 1$, by Theorem \ref{t:nonsolv}. 
\end{proof}

\begin{ex}\label{e:two} Notice that $minmax(L) - \ell(L)$ can take any value $n \in \N$. For, let $L = S_1 \oplus \ldots \oplus S_n$, where the $S_i$ are mutually non-isomorphic three-dimensional non-split simple Lie algebras for $1 \leq i \leq n$. Over the rational field there are infinitely many such $S_i$. Let $M$ be a maximal subalgebra of $L$. Then Lemma \ref{l:simple} implies that $M = S_1 \oplus \ldots \oplus \hat{S_i} \oplus \ldots \oplus S_n \oplus Fs_i$, for some $1 \leq i \leq n$, where $s_i \in S_i$ and $\hat{S_i}$ indicates a term that is missing from the direct sum. Similarly it is easy to see that any maximal subalgebra of $M$ is isomorphic to $S_1 \oplus \ldots \oplus \hat{S_i} \oplus \ldots \oplus S_n$ or to $$S_1 \oplus \ldots \oplus \hat{S_i} \oplus \ldots \oplus \hat{S_j} \oplus \ldots \oplus S_n \oplus Fs_i \oplus Fs_j.$$ It follows that $\ell(L) = n$ and $minmax(L) = 2n$.   
\end{ex}

\section{Chains of modular subalgebras and of quasi-ideals of maximal length}
We shall need the following characterisation of modular subalgebras as given by Amayo and Schwarz in \cite{as}.

\begin{theor}\label{t:mod} (\cite[page 311]{as})A modular subalgebra $M$ of a finite-dimensional Lie algebra $L$ over
any field of characteristic zero is either
\begin{itemize}
\item[(i)] an ideal of $L$; or
\item[(ii)] $L/M_L$ is almost abelian, every subalgebra of $L/M_L$ is a quasi-ideal, $M/M_L$ is
one-dimensional and is spanned by an element which acts as the identity map
on $([L, L]+M_L)/M_L$; and $L/([L, L] + M_L)$ is one-dimensional; or
\item[(iii)] $M/M_L$ is two-dimensional and $L/M_L$ is the three-dimensional split simple
Lie algebra; or
\item[(iv)] $M/M_L$ is a one-dimensional maximal subalgebra of $L/M_L$ and $L/M_L$ is
a three-dimensional non-split simple Lie algebra.
\end{itemize}
\end{theor}

\begin{theor}\label{t:modl} Let $L$ be a Lie algebra with radical $R$ over a field of characteristic zero. Then 
\begin{itemize}
\item[(i)] $\ell(L) \leq mod\ell(L) \leq \ell(L) + 1$; and
\item[(ii)] $mod\ell(L) = \ell(L) + 1$ if and only if $L/R$ has a three-dimensional simple ideal.
\end{itemize}
\end{theor}
\begin{proof} The fact that $\ell(L) \leq mod\ell(L)$ follows from the fact that ideals of $L$ are modular in $L$. Let
\begin{align}
0 = M_0 < M_1 < \ldots < M_s = L 
\end{align}
be a chain of modular subalgebras of $L$ of maximal length, and suppose that $M_i$ is the first ideal of $L$ that we encounter in going down the chain from $M_{s-1}$. Then $M_{i-k}$ is an ideal of $L$ for each $0 \leq k \leq i$, by \cite[Lemmas 1.1, 1.7]{tow}. If $i = s-1$ then $(2)$ is a chief series for $L$ and $\ell(L) = s = mod\ell(L)$. So suppose that $i < s-1$. Then $M_i = (M_{i+1})_L$ and one of cases (ii), (iii) and (iv) of Theorem \ref{t:mod} holds. If $L/M_i$ is almost abelian, then $\ell(L) = i+2 = s$, and so $mod\ell(L) = \ell(L)$. If $L/M_i$ is three-dimensional simple, then $\ell(L) = i+1 = s-1$ and so $mod\ell(L) = \ell(L) + 1$. This proves both (i) and (ii).  
\end{proof}
\bigskip

A straightforward corollary is the following result which gives a purely lattice-theoretic characterisation of certain algebras. 

\begin{coro}\label{c:modl} Let $L$ be a Lie algebra with radical $R$ over a field of characteristic zero. Then $mod\ell(L) = minmax(L)$ if and only if either
\begin{itemize}
\item[(i)] $L=R$, or 
\item[(ii)] $\sqrt{F} \not \subseteq F$ and $L/R$ is a direct sum of isomorphic three-dimensional simple Lie algebras.
\end{itemize}
\end{coro}
\begin{proof} Suppose first that $mod\ell(L) = minmax(L)$ but $L$ is nonsolvable. Then
$$ \ell(L) + 1 \leq minmax(L) = mod\ell(L) \leq \ell(L) + 1,$$
by Theorem \ref{t:nonsolv} and Theorem \ref{t:modl} (i). It follows that $minmax(L) = \ell(L) + 1$ and so $\sqrt{F} \not \subseteq F$ and $L/R$ is a direct sum of isomorphic three-dimensional simple Lie algebras, by Proposition \ref{p:split}.
\par

The converse follows from Theorems \ref{t:maxmin}, \ref{t:modl} and Proposition \ref{p:split}.
\end{proof}
\bigskip

We recall the definition of the algebras $L_m(\Gamma)$ over a field $F$ of characteristic zero or $p$, where $p$ is prime, as given by Amayo in \cite[page 46]{am2}. Let $m$ be a positive integer satisfying 
$$m = 1, \hspace{1cm} \hbox{or if } p \hbox{ is odd, } \hspace{.2cm} m = p^r - 2 \hspace{.2cm} (r \geq 1),$$ 
$$\hbox{or if } p =2, \hspace{.2cm} m = 2^r - 2 \hbox{ or } m = 2^r - 3 \hspace{.2cm} (r \geq 2).$$ 
Let $\Gamma = \{\gamma_0, \gamma_1, \ldots\} \subseteq F$ subject to 
$$(m + 1 - i) \gamma_i = \gamma_{m+i-1} = 0 \hspace{.3cm} \hbox{for all } i \geq 1, \hbox{ and}$$ $$\lambda_{i,k+1-i} \gamma_{k+1} = 0 \hspace{.3cm} \hbox{for all } i,k \hbox{ with } 1 \leq i \leq k.$$
Let $L_m(\Gamma)$ be the Lie algebra over $F$ with basis $v_{-1}, v_0, v_1, \ldots, v_m$ and products 
$$[v_{-1},v_i] = -[v_i,v_{-1}] = v_{i-1} + \gamma_i v_m, \hspace{1cm} [v_{-1}, v_{-1}] = 0,$$ 
$$[v_i,v_j] = \lambda_{ij} v_{i+j} \hbox{ for all } i,j \hbox{ with } 0 \leq i,j \leq m,$$
where $v_{m+1} = \ldots = v_{2m} = 0$.    
\par
We shall need the following classifications of quasi-ideals and of Lie algebras with core-free subalgebras of codimension one as given in \cite{am1} and \cite{am2}.

\begin{theor}\label{t:qi} (\cite[Theorem 3.6]{am1}) Let $Q$ be a core-free quasi-ideal of the Lie algebra $L$ over a field $F$. Then one of the following possibilities occurs.
\begin{itemize}
\item[(i)] $Q$ has codimension one in $L$;
\item[(ii)] $L \cong L_1(0)$ (defined below) and $F$ has characteristic two; or
\item[(iii)] $L$ is almost abelian and $Q$ is one-dimensional.
\end{itemize}
\end{theor}

\begin{theor}\label{t:am} (\cite[Theorem 3.1]{am2})
Let $L$ have a core-free subalgebra of codimension one. Then either (i) dim $L \leq 2$, or else (ii) $L \cong L_m(\Gamma)$ for some $m$ and $\Gamma$ satisfying the above conditions.
\end{theor}

We shall also need the following properties of $L_m(\Gamma)$ which are given by Amayo in \cite{am2}.

\begin{theor}\label{t:gamma} (\cite[Theorem 3.2]{am2})
\begin{itemize}
\item[(i)] If $m > 1$ and $m$ is odd, then $L_m(\Gamma)$ is simple and has only one subalgebra of codimension one.
\item[(ii)] If $m > 1$ and $m$ is even, then $L_m(\Gamma)$ has a unique proper ideal of codimension one, which is simple, and precisely one other subalgebra of codimension one.
\item[(iii)] $L_1(\Gamma)$ has a basis $\{u_{-1}, u_0, u_1 \}$ with multiplication $[u_{-1}, u_0] = u_{-1} + \gamma_0 u_1$ $(\gamma_0 \in F, \gamma_0 = 0$ if $\Gamma = \{0\})$, $[u_{-1}, u_1] = u_0, [u_0, u_1] = u_1$.
\item[(iv)] If $F$ has characteristic different from two then $L_1(\Gamma) \cong L_1(0) \cong sl_2(F)$.
\item[(v)] If $F$ has characteristic two then $L_1(\Gamma) \cong L_1(0)$ if and only if $\gamma_0$ is a square in $F$. 
\end{itemize}
\end{theor}

\begin{theor}\label{t:qil} Let $L$ be a Lie algebra over a field $F$. Then 
\begin{itemize}
\item[(i)] $\ell(L) \leq qi\ell(L) \leq \ell(L) + 2$;
\item[(ii)] $qi\ell(L) = \ell(L) + 2$ if and only if $F$ has characteristic two and $L$ has an ideal $B$ such that $L/B \cong L_1(0)$; and
\item[(iii)] $qi\ell(L) = \ell(L) + 1$ if and only if $L$ has an ideal $B$ such that $L/B \cong L_m(\Gamma)$ where $m$ is odd, and $\gamma_0$ is not a square in $F$ if $m=1$.
\end{itemize}
\end{theor}
\begin{proof} We proceed as in Theorem \ref{t:modl}. The fact that $\ell(L) \leq qi\ell(L)$ follows from the fact that ideals of $L$ are quasi-ideals of $L$. Let
\begin{align}
0 = M_0 < M_1 < \ldots < M_s = L 
\end{align}
be a chain of quasi-ideals of $L$ of maximal length, and suppose that $M_i$ is the first ideal of $L$ that we encounter in going down the chain from $M_{s-1}$. Then $M_{i-k}$ is an ideal of $L$ for each $0 \leq k \leq i$, by \cite[Lemmas 1.1, 1.7]{tow} and the fact that quasi-ideals of $L$ are modular in $L$. If $i = s-1$ then $(3)$ is a chief series for $L$ and $\ell(L) = s = qi\ell(L)$. So suppose that $i < s-1$. Then $M_i = (M_{i+1})_L$ and one of the cases of Theorems \ref{t:am} and \ref{t:qi} holds. If $L/M_i$ is almost abelian, then $\ell(L) = i+2 = s$, and so $qi\ell(L) = \ell(L)$. If not, then $L/M_i \cong L_m(\Gamma)$. There are the following possibilities:
\begin{itemize}
\item if $m>1$ and $m$ odd, then $\ell(L) = i+1 = s-1$, so $qi\ell(L) = \ell(L)+1$;
\item if $m>1$ and $m$ even, then $\ell(L) = i+2 = s$, so $qi\ell(L) = \ell(L)$;
\item if $m=1$ and $\gamma_0$ is not a square if $F$ has characteristic $2$, then $\ell(L) = i+1 = s-1$, so $qi\ell(L) = \ell(L)+1$; and finally
\item if $m=1$, $F$ has characteristic $2$ and $L_1(\Gamma) \cong L_1(0)$, then $\ell(L) = i+1 = s-2$, so $qi\ell(L) = \ell(L)+2$. 
\end{itemize}
This establishes all three cases.
\end{proof}

\begin{coro}\label{c:restricted} Let $L$ be a restricted Lie algebra over an algebraically closed field $F$ of characteristic $p > 0$. Then
\begin{itemize}
\item[(i)] $mod\ell(L) \leq \ell(L) + 1$; and
\item[(ii)] $mod\ell(L) = \ell(L) + 1$ if and only if $L$ has an ideal $B$ such that $L/B \cong sl_2(F)$ or the Witt algebra $W(1:\underline{1})$.
\end{itemize}
\end{coro}
\begin{proof} Since $L$ is restricted over an algebraically closed field $F$ of characteristic $p > 0$, every modular subalgebra of $L$ is a quasi-ideal of $L$; this was proved for $p>7$ by Varea in \cite[Theorem 2.2]{mod} and extended to cover $p=2,3,5,7$ by Towers in \cite[Theorem 2.2]{sm}. Moreover, case (ii) of Theorem \ref{t:qil} cannot occur, and case (iii) only occurs when $L/B \cong sl_2(F)$ or the Witt algebra $W(1:\underline{1})$ (see the last paragraph of the proof of \cite[Lemma 3.4]{am2}).
\end{proof}

\end{document}